\documentclass[11pt]{article}
\usepackage{amssymb, amsmath, amsthm, tikz, xspace,subfigure, graphicx}
\usepackage[left=1in,top=1in,right=1in]{geometry}
\usepackage{stmaryrd}

\usetikzlibrary{positioning, shapes.geometric, calc, intersections}
\tikzstyle{vertex}=[circle,fill=black,inner sep=2pt]
\tikzstyle{vertrect}=[draw,rectangle,inner sep=2pt]
\tikzstyle{vertdia}=[draw,diamond,inner sep=2pt]

\let\originalleft\left
\let\originalright\right
\renewcommand{\left}{\mathopen{}\mathclose\bgroup\originalleft}
\renewcommand{\right}{\aftergroup\egroup\originalright}

\newcommand{\paren}[1]{\left( #1 \right)}
\newcommand{\abs}[1]{\left| #1 \right|}
\newcommand{\brac}[1]{\left\{ #1 \right\}}

\newcommand{\ind}{\operatorname{ind}}

\theoremstyle{plain}
      \newtheorem{theorem}{Theorem}[section]
      \newtheorem{lemma}[theorem]{Lemma}
            \newtheorem{claim}[theorem]{Claim}

\theoremstyle{definition}
      \newtheorem{definition}[theorem]{Definition}

	\newcommand{\vep}{{\varepsilon}}
    \newcommand{\C}{{\mathcal{C}}}
\newcommand{\df}{\stackrel{\rm def}{=}}

\title{Inducibility of rainbow graphs}

\author{Emily Cairncross\thanks{Department of Mathematics, Statistics and Computer Science, University of Illinois Chicago, IL 60607. Email: emilyc10@uic.edu. Research partially supported by NSF Award
DMS-1952767.} \and Clayton Mizgerd\thanks{Department of Mathematics, Statistics and Computer Science, University of Illinois Chicago, IL 60607. Email: cmizge2@uic.edu. Research partially supported by NSF Award ECCS-2217023. } \and Dhruv Mubayi\thanks{Department of Mathematics, Statistics and Computer Science, University of Illinois Chicago, IL 60607. Email: mubayi@uic.edu. Research partially supported by NSF Awards 
DMS-1952767 and DMS-2153576, and by a Simons Fellowship.} }

\begin{document}

\maketitle

\begin{abstract}
Fix $k\ge 11$ and a rainbow $k$-clique $R$. We prove that the inducibility of $R$ is $k!/(k^k-k)$. An extremal construction is a balanced recursive blow-up of $R$.  This answers a question posed by Huang, that is a generalization of an old problem of Erd\H os and S\'os. It remains  open  to determine the minimum $k$ for which our result is true. 
More generally, we prove that there is an absolute constant $C>0$ such that every $k$-vertex connected rainbow graph  with minimum degree at least $C\log k$  has inducibility $k!/(k^k-k)$.
\end{abstract}

\section{Introduction}
Fix a graph $F$ on $k$ vertices and another graph $G$ on $n>k$ vertices. 
Write $I(F, G)$ for the number of $k$-subsets $S \subset V(G)$ 
such that $G[S] \cong F$ and let 
$$\varrho (F, G) := \frac{I(F, G) }{ \binom{n}{k}}.$$ Many  foundational questions in extremal graph theory deal with estimating $\varrho (F, G)$ for various choices of $F$ and $G$. One central question is to determine the minimum value when $F$ is a clique and $G$ has a specified edge density~\cite{Raz, N, R}, but there are also many fundamental questions about the maximum value regardless of edge density. This is the direction we take here.

Let $I(F, n)$ be the maximum  of 
$I(F, G)$ over all $n$ vertex graphs $G$. A standard averaging argument implies that
$$ \ind(F, n):= \frac{I(F, n)}{{\binom nk}} \le \frac{I(F, n-1)}{{\binom{n-1}{k}}} =\ind(F, n-1).$$

Thus, $\ind(F,n)$ is a decreasing sequence bounded below by zero, so it has a limit. Define the {\em inducibility} of $F$ to be  $$\operatorname{ind} (F) := \lim_{n\to \infty} \ind(F, n).$$ 

The {\em iterated balanced blow-up of a graph $F$} is a family $\mathcal{G}_F(n)$ of graphs on $n$ vertices defined inductively as follows.  Label $V(F)$ with $[k]:=\{1, \ldots, k\}$.  For $n < k$, the family $\mathcal{G}_F(n)$ contains only the empty graph on $n$ vertices.  For $n \geq k$, for any $G \in \mathcal{G}_F(n)$, we have a partition $V(G) = V_1 \cup \cdots \cup V_k$ with the following properties:
\begin{enumerate}
    \item For all $i,j \in [k]$, $\big| |V_i| - |V_j| \big| \leq 1$.
    \item For all $i \in [k]$, the induced subgraph $G[V_i] \in \mathcal{G}_F(|V_i|)$.
    \item For all $v \in V_i, w \in V_j$ with $i \ne j$, we have $vw \in E(G)$ if and only if $ij \in E(F)$.
\end{enumerate}
In many interesting cases, the construction above achieves the inducibility of $F$ and we now define this formally (our definition is slightly different than that in~\cite{LMP}).
\begin{definition}
    A graph $F$ is a {\em fractalizer} if 
  \[ \operatorname{ind} (F) = \lim_{n \to \infty} \max_{G \in \mathcal{G}_F(n)} \varrho \paren{F, G}.
    \]
In other words, the iterated balanced blow-up of $F$ achieves the inducibility.
       \end{definition}

The subgraph induced by every $k$-set comprising exactly one vertex in each $V_i$ is isomorphic to $F$. Consequently, for every $G \in \mathcal{G}_F(n)$,
$$I(F, G) \ge \sum_{i=1}^k I(F, G[V_i])+\prod_{i=1}^k|V_i|.$$
Together with a standard computation (see, e.g.~\cite{MR}), this yields
\begin{equation}\label{eqn:fractalizer}
    \operatorname{ind} (F) \geq \lim_{n \to \infty} \max_{G \in \mathcal{G}_F(n)} \varrho \paren{F, G} \ge \frac{k!}{k^k-k}.
\end{equation}
Hence, if $F$ is a fractalizer, then $\operatorname{ind}(F) \ge k!/(k^k-k)$.
In most cases we consider, the fact that $F$ is a fractalizer will imply further that $\operatorname{ind}(F)=k!/(k^k-k)$.

The fundamental conjecture in this area, due to Pippenger and Golumbic~\cite{PG}, states that for $k \ge 5$, the cycle $C_k$ is a fractalizer and satisfies $\operatorname{ind}(C_k)=k!/(k^k-k)$.
This conjecture has been resolved for $k=5$ by Balogh, Hu, Lidick\'y, and Pfender~\cite{BHLP} (see also~\cite{LMP}), but remains open for all $k \geq 6$. Kr\'al, Norin, and Volec~\cite{KNV} showed that  $I(C_k,n)\le 2n^k/k^k$.
More generally, Fox, Huang, and Lee~\cite{FHL} and Yuster~\cite{Y} independently proved that random graphs are fractalizers asymptotically almost surely.  Fox, Sauermann, and Wei~\cite{FSW} further proved that random Cayley graphs of abelian groups with small number of vertices removed are almost surely fractalizers.

We now consider these notions on colored and directed structures. A {\em tournament} is an orientation of a complete graph.  An edge-coloring of a graph or tournament $G$ is a function $\chi : E(G) \to T$ where  $T$ is a set of colors; we say that $G$ is $T$-colored. A colored graph or tournament $G$ is  {\em rainbow} if $\chi$ is injective. Two colored graphs (or tournaments) $G$ and $H$ are  {\em isomorphic}, written $G \cong H$, if there exists a bijection $\varphi : V(G) \to V(H)$ such that the colors (and orientations) of all edges are preserved under  $\varphi$. If $F$ is a colored tournament or colored complete graph, then $\operatorname{ind} (F)$ is defined identically as in the graph case, but with these altered definitions of graph isomorphism; naturally, the underlying graph $G$ should have the colors or orientations corresponding to $F$. If $F$ is an arbitrary colored graph, then we can color all missing edges with a single new color and view $F$ as a colored complete graph. Consequently, we can define {\em fractalizer} for all these structures.

There are very few results on the inducibility of colored, oriented, or directed structures. The first exact result which involved an iterated construction was due to Huang~\cite{Hstar} who determined the inducibility of the directed star. Later, in order to solve an old conjecture of Erd\H os and Hajnal~\cite{EH} in hypergraph Ramsey theory,  the third author and Razborov~\cite{MR} proved the following result for $k \ge 4$ (the case that $k=3$ was proven earlier by Conlon, Fox, and Sudakov~\cite{CFS}).

\begin{theorem}[\cite{MR}]\label{thm:MR}
    All rainbow tournaments $R$ on $k \ge 4$ vertices are fractalizers.  In particular, $\mathrm{ind}(R) = k!/(k^k-k)$.
\end{theorem}

In this paper, we consider the question addressed by Theorem~\ref{thm:MR} in the undirected setting.  The first conjecture in this setting is due to Erd\H os and S\'os from the 1970s (see~\cite[Equation (20)]{EH}), and implies, in particular, that a rainbow triangle is not a fractalizer. Their conjecture was proved by Balogh et.~al.~\cite{B}, who showed that a blow-up of a properly $3$-edge-colored $K_4$ (instead of a rainbow $K_3$) achieves the inducibility of the rainbow triangle.  See also~\cite{CY} for similar computations, but in terms of the number of edges of each color instead of the number of vertices.  

Huang~\cite{Huang} asked whether Theorem~\ref{thm:MR} can be extended to the undirected setting for cliques of size larger than three. This, in particular, would imply that the phenomenon conjectured by Erd\H os and S\'os and proved in~\cite{B} (that $K_k$ is not a fractalizer for $k=3$)  fails to hold for larger $k$.  Our first result  addresses Huang's question and proves that rainbow $K_k$ are fractalizers for $k \geq 11$.

\begin{theorem}\label{thm:main1}
    All rainbow cliques $R$ on $k \ge 11$ vertices are fractalizers.  In particular,
    \[ \mathrm{ind}(R) = \frac{k!}{k^k-k}. \]
\end{theorem}

We make the following observations regarding Theorem~\ref{thm:main1}.
\begin{itemize}
    
        \item Theorem~\ref{thm:main1} implies Theorem~\ref{thm:MR} for $k \geq 11$, since any construction of a tournament inducing $\ell$ rainbow copies of $R$ yields a corresponding construction of a complete graph that induces at least $\ell$ rainbow (undirected) copies of $R$  by ignoring orientations.

        \item Similarly, if a graph $G$ on $k$ vertices is known to have inducibility $k! /( k^k - k)$, then the rainbow $k$-clique is a fractalizer as well by the following argument. Let $R$ be a rainbow $k$-clique and let $e_1, e_2, \ldots, e_m \in E(R)$ such that $(V(R), \{ e_1, \ldots, e_m \} )$ is a rainbow copy of $G$. Let $c_1, \ldots, c_m$ be the colors assigned to $e_1, \ldots, e_m$, respectively. Then any construction of an edge-colored graph inducing $\ell$ rainbow copies of $R$ induces at least $\ell$ copies of $G$ by deleting all edges except those colored by $c_1, \ldots, c_m$ and then ignoring the edge colors. It follows that $\ind (R) \le \ind(G) = k!/ (k^k -k)$, so $R$ is a fractalizer. Thus, the result of~\cite{BHLP} that $C_5$ is a fractalizer with $\ind (C_5) = 5!/(5^5 - 5)$ implies that the rainbow $5$-clique is a fractalizer, and the result of~\cite{FHL} that random graphs are almost surely fractalizers implies that the rainbow $k$-clique is a fractalizer for large $k$.

        \item We believe our proof of Theorem~\ref{thm:main1} has been optimized and requires $k \ge 11$. As~\cite{BHLP} showed that rainbow $5$-cliques are fractalizers, and~\cite{B} showed that rainbow $3$-cliques are not fractalizers, it remains open to determine whether rainbow $k_0$-cliques are fractalizers only for $k_0\in \{4,6,7,8,9,10\}$. 
        
\end{itemize}

Our proof of Theorem~\ref{thm:main1} follows the broad framework of the proof of Theorem~\ref{thm:MR} but there are several nontrivial technical difficulties that need to be addressed in the undirected setting. The difficulties arise due to the following reason:  the role that each endpoint of an edge plays in a rainbow copy of a tournament is  determined by the color and orientation of the edge, but this is no longer true in the undirected setting. 
We overcome these obstacles by adding some new ideas, at the expense of requiring a slightly higher value of $k$. For example, our proof of Theorem~\ref{thm:main1} requires a bound on the color degree of a vertex  and this was not needed in~\cite{MR}.

For large values of $k$, we prove the following more general result which shows that the analog of Theorem~\ref{thm:main1} holds for much sparser graphs. The proof requires several major new ideas.

\begin{theorem}\label{thm:main2}
     There exists an absolute constant $C >0$ such that all connected rainbow graphs $R$ with $k$ vertices and minimum degree at least $C \log k$ are fractalizers and satisfy \[
    \mathrm{ind}(R) = \frac{k!}{k^k-k}.
    \]
\end{theorem}

We make the following observations regarding Theorem~\ref{thm:main2}.
\begin{itemize}
        \item Theorem~\ref{thm:main2} implies Theorem~\ref{thm:main1} for large  $k$, since $R$ may be viewed as a rainbow $k$-clique with edges deleted. Let $c_1, c_2, \dots, c_m$ be the colors assigned to the deleted edges. Any construction of a colored complete graph inducing $\ell$ rainbow copies of the rainbow $k$-clique yields at least $\ell$ rainbow copies of $R$ by deleting edges colored $c_1, \ldots, c_m$.

        \item The requirement that $R$ is connected in the statement of Theorem~\ref{thm:main2} is necessary, as disconnected rainbow graphs without isolated vertices are not fractalizers (see Section~\ref{sec:matching}).

        \item We are not able to show that our requirement on minimum degree is tight, and this remains open.
        
    \end{itemize}




Theorem $\ref{thm:main1}$ is proven in Section~\ref{sec:rainbow_clique} and Theorem $\ref{thm:main2}$ is proven in Section~\ref{sec:rainbow_graph}. In Section~\ref{sec:matching}, we justify the second observation above.

\section{Proof of Theorem $\ref{thm:main1}$}\label{sec:rainbow_clique}

We give the proof of Theorem~\ref{thm:main1} in the following subsections.

\subsection{Setup}\label{sec:setup1}

Fix $k \ge 11$ and $T = \binom{[k]}{2}$. Let $R$ be a $T$-colored rainbow $k$-clique with coloring function $\chi_R$ and for concreteness, put $V(R):=[k]$  and $\chi_R (ij) = \{i,j\}$ for all $i,j \in [k]$.  

Set 
$$a := \frac{k!}{k^k-k}.$$
Our goal is to prove that $\ind(R) \le a$. To this end, fix $\gamma>0$ and assume for contradiction $\ind(R)=a+\gamma$. Next choose $0 < \vep < \min\{\gamma,\ind(R)\}/100$.  Let $c_0$ be chosen so that  $\ind(R, n) \le \ind(R)+\varepsilon$ for all $n>c_0$. Choose $n_0 \ge \left\lceil 2k!c_0/\vep \right \rceil$ such that \begin{equation}\label{eq:nk_nk}
    \frac{n^k}{(n)_k} < 1 + \vep
\end{equation}
and
\begin{equation}\label{eq:n0}
    a \paren{\frac{n^{k - 1}}{(k - 1)!} - \binom{n - 1}{k - 1}} < \gamma \binom{n - 1}{k - 1} - {\binom{n - 2}{k - 2}}
\end{equation} 
for all $n>n_0$. This is possible since $\lim_{n \to \infty} n^k/(n)_k = 1$ and $n^{k - 1}/(k - 1)! - \binom{n - 1}{k - 1} = O(n^{k - 2})$, while $\gamma \binom{n - 1}{k - 1} - {\binom{n - 2}{k - 2}} = \Omega \paren{n^{k - 1}}$ as  $n \rightarrow \infty$.
Suppose that $n>n_0$ is given and $H$ is a $T$-colored $n$-vertex graph with coloring function $\chi_H$ achieving $I(R,n)$. This implies 
$$I(R, H) = I(R, n) = \ind(R,n){\binom nk}$$
where $a+\gamma =\ind(R) \le \ind(R,n) \le \ind(R) + \varepsilon = a+\gamma + \varepsilon$.

\begin{definition}\label{def:pq}
    For $q\ge 0$ and $t > 0$, let $p(q,t)$ be the maximum of $\prod_i q_i$ where $q_1 + \cdots + q_t = q$ and each $q_i \ge 0$ an integer.
\end{definition}

The AMGM inequality yields $p(q,t) \le (q/t)^t$ and it is easy to see that
   \begin{equation}\label{eqn:pq}
    p(q, t) p(q', t') \le p(q + q', t + t')
\end{equation} for all $q, q' \ge 0$ and $t, t'>0$ (see Appendix).

For a vertex $x$ in $V(H)$ and $i \in [k]$, write $d_i(x)$ for the number of copies of $R$ containing $x$ where $x$ plays the role of vertex $i$ in $R$.
More formally, $d_i(x)$ is the number of isomorphic embeddings $\phi: R \to H$ such that $\phi(i) = x$.
Let $d(x)=\sum_i d_i(x)$ be the number of copies of $R$ containing $x$.  We will refer to this as the \emph{degree of $x$ in $H$}. Similarly, let $d(x,y)$ be the number of copies of $R$ containing both $x$ and $y$. For $i \in [k]$, let $N_i(x)$ be the set of $y \in V(H)\setminus \{x\}$ for which there is a copy of $R$ in $H$ containing both $x$ and $y$ in which $x$ plays the role of vertex $i$ in $R$. Note that we do not have $N_j(x) \cap N_{j'}(x) =\emptyset$ for $j \ne j'$, but all edges between $N_j(x) \cap N_{j'}(x)$ and $x$ have the same color. However, $N_i(x)$ has a (unique) partition $\cup_{j \ne i} N_i^j(x)$ where $N_i^j(x)$ comprises those $y$ such that $x,y$ lie in a copy of $R$ with $x$ playing the role of $i$ and $y$ playing the role of $j$. Indeed, the partition is obtained based on the color of a vertex to $x$. This gives
\begin{equation} \label{eqn:dn1}
	d(x) = \sum_{i=1}^k d_i(x) \le \sum_{i=1}^k \prod_{j \ne i} |N_i^j(x)| \le \sum_{i=1}^k  p(|N_i(x)|, k-1).
\end{equation}
We partition $V(H)$ into $V_1 \cup \cdots \cup V_k$, where
$$V_i = \{x \in V(H):  |N_i(x)| \ge |N_j(x)| \hbox{ for all } j \ne i\}.
$$
If there is a tie, we break it arbitrarily. Set $n_i=|V_i|$ for all $i \in [k]$.

\subsection{Minimum degree}\label{sec:mindeg1}
Here we show that a standard technique in extremal graph theory can be used to prove that 
each vertex of $H$ lies in at least  the average number of copies of $R$ (apart from a small error term).

\begin{lemma}\label{lem:min_deg1}
    $d(x) \geq a n^{k-1} / (k-1)!$ for all $x \in V(H)$.
\end{lemma}

\begin{proof}
    We write $d=b\pm c$ for the inequalities $b-c\le d \le b+c$. Denote the average degree of $H$ by
    $$d(H) := \frac{ k\cdot I(R,H)}{n} = \ind(R,n){\binom{n-1}{k-1}}.$$ 
    We claim that for every $x \in V(H)$
    \begin{equation} \label{eqn:regular1} d(x) = d(H) \pm  {\binom{n-2}{k-2}}.
    \end{equation}
    
    This follows from a standard application of Zykov symmetrization. Indeed, if the degrees of two vertices $x$ and $y$ differ by more than ${\binom{n - 2}{k - 2}}$, say $d(x) >d(y) + {\binom{n - 2}{k - 2}}$, then we can delete $y$ and duplicate $x$, meaning we add a new vertex $x'$ with $\chi_H (x'z) = \chi_H (xz)$ for all other vertices $z$, and $\chi_H (xx')$ can be arbitrary. This transformation increases the number of copies of $R$ by at least  
    $$d(x)-d(y) -d(x,y)\ge d(x)-d(y) - {\binom{n - 2}{k - 2}} >0, $$ contradicting the maximality $I(R,H) = I(R,n)$. Hence all degrees lie in an interval of length at most ${\binom{n - 2}{k - 2}}$ and (\ref{eqn:regular1}) follows, since this interval must contain $d(H)$. 
    In particular, the minimum degree is at least 
    $$d(H) - {\binom{n - 2}{k - 2}} = \ind(R,n){\binom{n - 1}{k - 1}} - {\binom{n - 2}{k - 2}}  \ge (a+\gamma){\binom{n - 1}{k - 1}} - {\binom{n - 2}{k - 2}}> 
    a\frac{n^{k-1}}{(k-1)!}$$ 
    for $n > n_0$. The last inequality follows from (\ref{eq:n0}).
\end{proof}

\subsection{Maximum color degree}

Let $$\alpha := \frac{\max_{x,i,j} d_{\{i,j\}}(x)}{n}$$
 where the maximum is taken over all vertices $x\in V(H)$ and all colors $\{i,j\} \in T$ and $d_{\{i,j\}}(x)$ is the number of edges in $H$ incident with $x$ in color $\{i,j\}$. We upper bound this value.

\begin{lemma}\label{lem:max_color_deg1}
    $\alpha \leq 0.4$.
\end{lemma}

\begin{proof}
    Let $x, i,j$ achieve this maximum, so that $d_{\{i,j\}}(x) = \alpha n$.
    Then $|N_i^j(x)| \leq \alpha n$ and $N_i(x)$ has a partition $N_{i}^j(x) \bigcup \cup_{\ell\ne j} N_i^{\ell}(x)$ where every copy of $R$ containing $x$ with $x$ playing the role of $i$ has exactly one vertex in each $N_i^{\ell}(x)$ for all $\ell \in [k]\setminus\{i\}$. Further, $$\abs{\bigcup_{\ell\ne j} N_i^{\ell}(x)} \leq n-d_{\{i,j\}}(x)$$ since a vertex incident to an edge colored $\{i, j \}$ cannot play the role of $\ell \neq i,j$.
    Consequently,
    $$d_i(x) \le   |N_{i}^j(x)| \cdot p \paren{n- d_{\{i,j\}}(x), k-2} \le \alpha\, n \cdot \left(\frac{(1-\alpha)n}{k-2}\right)^{k-2}.$$
    The same upper bound holds for $d_j(x)$. For $\ell \not\in \{i,j\}$, we have $N_{\ell} (x) \leq n-d_{\{i,j\}}(x)$ since $x$ is playing the role of $\ell$, so we cannot include an edge incident to $x$ of color $\{i,j\}$ since the color must include $\ell$. Hence
    $$d_{\ell}(x) \le p \paren{ n-d_{\{i,j\}}(x), k-1 } \le \left(\frac{(1-\alpha)n}{k-1}\right)^{k-1}
    \le \left(\frac{(1-\alpha)n}{k-2}\right)^{k-1}.$$
    Altogether this yields
    $$ d(x) \le 2 \,\alpha\, n \left(\frac{(1-\alpha)n}{k-2}\right)^{k-2} + (k-2)\left(\frac{(1-\alpha)n}{k-2}\right)^{k-1} =(1 + \alpha)\left(\frac{1 - \alpha}{k-2}\right)^{k-2}n^{k-1}.$$
    Suppose for contradiction that $\alpha>0.4$. Since  $k \ge 3$,  $(1+\alpha)(1-\alpha)^{k-2}$ is a decreasing function of $\alpha$ for $\alpha \in (0.4, 1]$,
    and $d(x) \geq an^{k-1}/(k-1)!$ by Lemma~\ref{lem:min_deg1}. Therefore \begin{equation} \label{eqn:max1}
        \frac{1}{k^{k-1}-1}= \frac{a}{(k-1)!} \leq \frac{d(x)}{n^{k-1}} \le 1.4 \left(\frac{0.6}{k-2}\right)^{k-2}.
    \end{equation}
    However, this fails to hold for $k \ge 11$ (see Appendix), and  we conclude that $\alpha \leq 0.4$ as desired.
\end{proof}

\subsection{The second largest neighborhood}

For a vertex $x \in V(H)$, let $Z(x)$ be the second largest set 
in  $\{N_1(x), \ldots, N_k(x)\}$ and define 
$$z:= z_{k,n} = \max_{x \in V(H)} \frac{|Z(x)|
}{n}.$$

\begin{lemma}\label{lem:nbhd_z1}
    $z \leq 0.5$.
\end{lemma}

\begin{proof}
    Let $x$ be such that $z=|Z(x)|/n$. Let $a_i=|N_i(x)|/n$ and assume by relabeling that $a_1 \ge a_2=z \ge a_3 \ge \cdots \ge a_k$.  Since $N_j(x) \cap N_{j'}(x) \cap N_{j''}(x) = \emptyset$ for any three distinct $j,j', j''$
    we have $\sum a_i \le 2$.  Let $a_3+\cdots +a_k= s\le 2-(a_1+z)$. Write $s=qz+r$ where $q \in \mathbb Z^{\ge 0}$ and $0\le r<z$. If $x \le y$, then $x^{k-1}+y^{k-1} < (x-\rho)^{k-1}+(y+\rho)^{k-1}$ for $0<\rho<x$ by convexity of $x^{k-1}$ so  successively increasing the largest $a_i$ to $z$ and decreasing the smallest $a_j$ to 0 or $r$, we obtain
    $$\sum_{i=3}^k a_i^{k-1} \le q z^{k-1} +r^{k-1}\le 
    qz^{k-1} + \frac{r}{z} z^{k-1} = \frac{s}{z}  z^{k-1} \le \frac{2-(a_1+z)}{z} z^{k-1}.$$
    Consequently, 
    $$\sum_{i=1}^k a_i^{k-1} = a_1^{k-1} + z^{k-1} + \sum_{i=3}^k a_i^{k-1} \le a_1^{k-1} + z^{k-1} + \frac{2-(a_1+z)}{z} z^{k-1}.$$
    Since $a_1 \geq z$, taking the derivative shows that for any $z$, this expression is increasing with $a_1$. 
    Using Lemma~\ref{lem:max_color_deg1}, we note that $a_1 + z \le 1+\alpha < 1.4$ since $$|N_1(x)|+|Z(x)|
    =|N_1(x) \cup Z(x)|+|N_1(x) \cap Z(x)| \le n + d_{\brac{1,2}} (x) \leq n +\alpha n< 1.4 \cdot n.$$  Thus $a_1 < 1.4-z$ and $a_1 \leq 1$, so
    \[ \sum_{i=1}^k a_i^{k-1} \leq a_1^{k-1} + z^{k-1} + \frac{2-(a_1+z)}{z} z^{k-1} \leq (\min\{1.4-z,1\})^{k-1} + z^{k-1} + \frac{0.6}{z} z^{k-1}. \]
    
    Using (\ref{eqn:dn1}) and Lemma~\ref{lem:min_deg1} yields
      $$\frac{1}{k^{k-1}-1} \leq \frac{d(x)}{n^{k-1}}  \le \sum_{i=1}^k \left(\frac{a_i}{k-1}\right)^{k-1} \le \frac{1}{(k-1)^{k-1}} \paren{(\min\{1.4-z,1\})^{k-1} + z^{k-1} + \frac{0.6}{z} z^{k-1}}.$$

    Multiplying by $(k-1)^{k-1}$ and using the fact that \begin{equation} \label{eq:e_inverse1}
        \frac{(k-1)^{k-1}}{k^{k-1}-1} \geq \frac{(k-1)^{k-1}}{k^{k-1}} = \paren{1-\frac1k}^{k-1}>\frac{1}{e}
        \end{equation}
        for $k>1$, we obtain
    $$\frac1e < (\min\{1.4-z,1\})^{k - 1} + z^{k-1} + 0.6 \cdot z^{k-2}.$$ 
    
    As $z \leq a_1$ and $z + a_1 < 1.4$, we have $z < 0.7$.
    Thus, the RHS is nonincreasing with $k$ and we may consider only the $k=11$ case.  Numerical calculations show that for $z \in [0.5,0.7]$, we have $(1.4-z)^{10} + z^{10} + 0.6z^9 < 1/e$, so we conclude that $z < 0.5$.
\end{proof}

\subsection{One large part}

We now take care of the situation when one of the $V_i$'s is very large.

\begin{lemma}\label{lem:large_V1}
    $|V_i| \le (1-1/3k)n$ for all $i \in [k]$.
\end{lemma}

\begin{proof}
    By contradiction, WLOG suppose that $|V_1|>(1-1/3k)n$. If $x \in V_1$, then $|N_1(x)|\ge |N_i(x)|$ for all $i>1$ so
    $|N_2(x)| \le |Z(x)| \le z n$.  Using (\ref{eqn:dn1}) we have 
    $$a{\binom{n}{k}} \le I(R, H) = \sum_{x \in V(H)} d_2(x) \le |V_1|p(zn, k-1)+ \frac{n}{3k}  p(n,k-1)<\left(z^{k-1}+\frac{1}{3k}\right) \frac{n^k}{(k-1)^{k-1}}.$$ 
    Using our lower bound on $n_0$ in 
    (\ref{eq:nk_nk}), we get \begin{equation}\label{eqn:large1}
        \paren{\frac{(k-1)^{k-1}}{k^k-k}} < (1+\vep) \paren{z^{k-1} + \frac{1}{3k}} .
    \end{equation} 
    This fails to hold for $k \ge 11$ (see Appendix).    
    We conclude that $|V_i| \le (1-1/3k)n$ for all $i \in [k]$. \end{proof}

\subsection{Counting the copies of $R$ in $H$}\label{sec:counting1}
Here we describe the broad framework we will use to count copies of $R$ in $H$. This is the same as in~\cite{MR}, though  there are subtle differences which arise since we are in the undirected setting.

Call a copy $f$ of $R$ in $H$ {\em transversal} if it includes exactly one vertex in $V_i$ for all $i \in [k]$. We partition the copies of $R$ in  $H$  as $H_m \cup H_g \cup H_b$ where $H_m$ comprises those copies that lie entirely inside some $V_i$, $H_g$ comprises those copies that intersect every $V_i$ whose edge coloring coincides with the natural one given by the vertex partition (meaning the map from $R$ to $H$ takes vertex $i$ to a vertex in $V_i$), and $H_b$ comprises all other copies of $R$ (these include transversal copies, but some vertex in any such copy will be in an inappropriate $V_i$). Let $h_m=|H_m|, h_g=|H_g|$ and $h_b=|H_b|$ so that
$$I(R, H) = h_m+h_g+h_b.$$
We will bound each of these three terms separately. First, note that 
\begin{equation} \label{eqn:hmbound1}
	h_m = \sum_j I(R, H[V_j]) \le \sum_j  I(R, n_j).
\end{equation}
Next we turn to $h_g$.
Let $\Delta$ denote the number of $k$-sets that intersect each $V_i$ but are not counted by $h_g$. So a $k$-set counted by $\Delta$  either does not form a copy of $R$, or forms a copy of $R$ but its edge coloring does not coincide with the natural one given by the vertex partition $V_1 \cup \ldots \cup
V_k$.
Then
\begin{equation} \label{eqn:hgbound1}
	h_g =\prod_i n_i -\Delta
\end{equation}
and we need to bound $\Delta$ from below.

Note that the color  of some pair in every member of $\Delta$ does not align with the implicit one given by our partition. With this in mind, let $D_{ij}$ be the set of pairs of vertices $\{v_i, v_j\}$ where $v_i \in V_i, v_j \in V_j$, $i\ne j$
such that $\chi_H (v_iv_j) \ne \chi_R (ij) = \{i,j\}$.
Let $\delta_{ij} =|D_{ij}|/{\binom{n}{2}}$,  $D = \cup_{ij} D_{ij}$ and
$\delta= |D|/{\binom{n}{2}}$. Let us  lower bound $\Delta$ by counting the misaligned
pairs from $D$ and then choosing the remaining $k-2$ vertices, one from each of the remaining parts $V_{\ell}$. This gives, for each $i<j$,
$$\Delta \ge |D_{ij}| \prod_{\ell \ne i,j} n_{\ell}=  \delta_{ij} {\binom{n}{2}} \prod_{\ell \ne i,j} n_{\ell}= \delta_{ij} {\binom{n}{2}} \frac{\prod_{\ell=1}^k n_{\ell}}{n_in_j}.$$
Since $\sum_{ij} \delta_{ij} {\binom{n}{2}} = \sum_{ij}|D_{ij}|=|D|=\delta{\binom{n}{2}},$
we obtain by summing over $i,j$,
$$\Delta\left(\sum_{1\le i<j\le k} n_in_j\right) \ge \delta {\binom{n}{2}} \prod_{\ell=1}^k n_{\ell}.$$ This with along with (\ref{eqn:hgbound1}) gives
\begin{equation} \label{eqn:delta+1}
	h_g \le\prod_{\ell=1}^k n_{\ell}\left(1-\frac{\delta{\binom{n}{2}}}{\sum_{1\le i<j\le k}n_in_j}\right) = \prod_{\ell=1}^k n_{\ell}\left(1-\frac{\delta{\binom{n}{2}}}{{\binom{n}{2}} - \sum_i {\binom{n_i}{2}}}\right).
\end{equation}
Our next task is to upper bound $h_b$. For a vertex $x$ and $j \in [k]$, recall that $N_j(x) \subset V(H)$ is the set of  $y$ such that $x,y$ lie in a copy of $R$ with $x$ playing the role of vertex $j$ in $R$.
Let us enumerate the set $J$ of tuples $(v,w,f)$ where $e=\{v,w\}\in D, f \in H_b$, $e \subset f$, and  $v \in V_i$, but $i \notin \chi_H (vw)$. This means that $v$ must play the role of $i'$ in $f$ for some $i' \neq i$, so the colors on all $k-1$ pairs $(v,x)$ with $x \in f$ contain 
$i'$; in particular $v$ is incident to $k-2$ pairs in $f$ whose color does not contain $i$. 
If $v\in V_i$ and $w \in V_j$, then say that $(v,w,f)$ is 1-sided if $|\chi_H (vw) \cap \{i,j\}|=1$ and $(v,w,f)$ is 2-sided if $|\chi_H (vw) \cap \{i,j\}|=0$.

Let $J_i$ be the set of $i$-sided tuples ($i=1,2$). We consider the weighted sum
$$S = 2|J_1| + |J_2|.$$
Observe that each $f \in H_b$ contains at least $k-2$ pairs from $D$. Indeed, if $f$ is transversal, then it must contain a miscolored vertex which yields at least $k-2$ pairs from $D$ in $f$. If $f$ is not transversal, then take a largest color class $\C$ of $f$ and observe that at least $|\C|-1$ of the vertices in $\C$ are miscolored. Also, note that $2 \leq |\C| \leq k-1$ since $f$ is not contained in one color class and we have assumed $f$ is not transversal.

Let $\C$ be the color class corresponding to color $j$. If exactly $|\C|-1$ vertices in $\C$ are miscolored, then every edge $vw$ where $v\in \C$ is miscolored and $w\in f \setminus \C$ is in $D$. Since $|f \setminus \C| = k - |\C|$, this yields at least $(|\C|-1)(k-|\C|) \ge k-2$ pairs from $D$ in $f$. On the other hand, if all $|\C|$ vertices in $\C$ are miscolored, then there is a unique vertex $u \in f \setminus \C$ that plays the role of $j$ in $f$. Every edge $vw$ where $v \in \C$ and $w \in f \setminus \paren{\C \cup u}$ is in $D$, so if $|\C| \leq k-2$, this yields at least $|\C| (k - |\C| - 1) \geq k-2$ pairs from $D$ in $f$. If $|\C| = k-1$, then $f = \C \cup u$ where $u$ plays vertex $j$ in $f$ but is in a different color class, say the color class corresponding to color $\ell$. There are $k-1$ edges between $\C$ and $u$, but only one can contain both $k$ and $\ell$, so at least $k-2$ edges from $D$ are in $f$.

We conclude that each $f \in H_b$  contributes at least $2(k-2)$ to $S$ since $f$ contains at least $k-2$ pairs $e=\{v,w\}\in D$
and if $(v,w,f)$ is 1-sided it contributes 2 to $S$ while if it is 2-sided then it contributes 2 again since both $(v,w,f)$
and $(w,v, f)$ are counted with coefficient 1. This yields
\begin{equation} \label{eqn:slower} S \ge 2(k-2)h_b.\end{equation}
On the other hand, we can bound $S$ from above by first choosing $e\in D$ and then $f \in H_b$ as follows. Call $v \in e =\{v,w\} \in D$ {\em correct in} $e$ if $v \in V_i$,  and $i \in \chi_H (vw)$; if $v$ is not correct in $e$ then $i \not\in \chi_H (vw)$ and say that $v$ is {\em wrong  in} $e$. The definition of $D$ implies that every $e\in D$ has at least one wrong vertex in $e$ (and possibly two wrong vertices). Let
$$D_i = \{\{v, w\}\in D: \{v, w\}\ \hbox{contains exactly $i$ wrong vertices}\} \qquad \hbox{($i=1, 2$)}.$$
The crucial observation is that
\begin{equation} \label{eqn:isided}( v,w,f) \in J_i \qquad \Longrightarrow \qquad \{v,w\} \in D_i \qquad \qquad \hbox{($i=1, 2$)}.\end{equation}

To bound $S$ from above, we use (\ref{eqn:isided}) and consider first $J_1$ and $D_1$.  We start by choosing $vw$ in $D_1$ with wrong vertex $v$. Note that $w$ is correct in $vw$ since $vw \in D_1$. Let $v \in V_i, w \in V_j$.  Then $\chi_H (vw) = \{j,\ell\}$ for some $\ell \ne i$ since $v$ is wrong in $e$ but $w$ is correct in $e$.  Thus for each triple $(v,w,f) \in J_1$, vertex $v$ plays the role of $j$ in $f$ or $v$ plays the role of $\ell$ in $f$; thus the total number of $(v,w, f) \in J_1$ for some $f$ is at most $p(|N_j(v)|-1,k-2) + p(|N_\ell(v)|-1,k-2)$.  Summing over all $vw \in D_1$, we get 
\[ |J_1| \leq \sum_{vw \in D_1} p(|N_j(v)|-1,k-2) + p(|N_\ell(v)|-1,k-2) \leq 2|D_1| p(zn,k-2). \]

The bound for  $J_2$ is similar.  Choose $vw \in D_2$ with $v \in V_i, w \in V_j$.  Let $\chi_H (vw) = \{\ell_1,\ell_2\}$ where $\{\ell_1, \ell_2\} \cap \{i,j\} = \emptyset$.  Since $vw$ is two-sided, we see that $(v,w,f)\in J_2$ exactly when $(w,v,f) \in J_2$.  Consequently,
\begin{align*}
    |J_2| & \leq \sum_{vw \in D_2} p(|N_{\ell_1}(v)|-1,k-2) + p(|N_{\ell_2}(v)|-1,k-2) \\ & \hspace{30pt} + p(|N_{\ell_1}(w)|-1,k-2) + p(|N_{\ell_2}(w)|-1,k-2) \\[.2 cm] & \leq 4 |D_2| p(zn,k-2).
\end{align*}

This gives
\begin{equation} \label{eqn:supper}
	 S = 2 |J_1| + |J_2|\le 4 \, |D| \, p(zn, k-2)\le 4 \, \delta{\binom{n}{2}}\left(\frac{z}{k-2}\right)^{k-2} n^{k-2}.\end{equation}
Finally, (\ref{eqn:slower}) and (\ref{eqn:supper}) give
    \begin{equation} \label{eqn:delta-1}h_b \le \frac{S}{2(k-2)} \leq \frac{ 2 \delta{\binom{n}{2}}}{k-2}\left(\frac{z}{k-2}\right)^{k-2}n^{k-2}.
\end{equation}
Using (\ref{eqn:hmbound1}), (\ref{eqn:delta+1}) and (\ref{eqn:delta-1})  we have that \begin{equation}\label{eqn:bound1}
    I(R, n) \leq \sum_{i} I(R, n_i) +
    \prod_{\ell} n_{\ell}\left(1-\frac{\delta{\binom{n}{2}}}{{\binom{n}{2}} - \sum_i {\binom{n_i}{2}}}\right)
    + \frac{ 2\delta{\binom{n}{2}}}{k-2}\left(\frac{z}{k-2}\right)^{k-2}n^{k-2}.
\end{equation}
Our final task is to upper bound the RHS.

Since  $\delta{\binom{n}{2}} \le \sum_{i \neq j}n_in_j = {\binom{n}{2}} - \sum_i {\binom{n_i}{2}}$, we have $\delta \in I\df \left[0, 1 - \sum_i {\binom{n_i}{2}}/{\binom{n}{2}} \right]$. Viewing (\ref{eqn:bound1}) as a linear function of $\delta$, it suffices to check the endpoints of $I$.

\subsection{The extremal case}\label{sec:extremal1}

\begin{claim}
    If $\delta = 0$, then $\ind(R) \le a$.
\end{claim}

\begin{proof}
    If $\delta=0$, then (\ref{eqn:bound1}) implies that
    \begin{equation} \label{eqn:I(R,n)case1_1}
    	I(R, n) \le \sum_{i=1}^k I(R, n_i) + \prod_{i=1}^k n_i.\end{equation}
    
    Let $p_i := n_i/n$.  Using $\max_i p_i \le 1-1/3k$ by Lemma~\ref{lem:large_V1}, convexity of $x^k$, and $k \geq 11$ we obtain
    \begin{equation} \label{eqn:psumbound1}
        \sum_{i=1}^k p_i^k \le \left(1-\frac{1}{3k}\right)^k + \left(\frac{1}{3k}\right)^k\le e^{-1/3} + 33^{-11} < 0.72.
    \end{equation}
    
    We begin by bounding the summation in (\ref{eqn:I(R,n)case1_1}).  By relabeling if necessary, let $n_1\le \cdots \le n_{\ell} \le c_0<n_{\ell+1}\le \cdots \le n_k$ where $\ell \ge 0$.  We have that
    \begin{equation} \label{eqn:pin1} \sum_{i=1}^k I(R, n_i)\leq \ell{\binom{c_0}{k}}+\sum_{i=\ell+1}^k I(R, n_i)\le
    \ell{\binom{c_0}{k}}+(\ind(R)+\varepsilon)\sum_{i=\ell+1}^k {\binom{n_i}{k}} .\end{equation}
    
     Observe that ${\binom{n_i}{k}}={\binom{p_in}{k}} < p_i^k{\binom{n}{k}}$ since $p_i<1$. Dividing (\ref{eqn:pin1}) by $\binom nk$ yields
    \begin{equation}\label{eqn:sumIbound1}
        \frac{1}{\binom nk} \sum_{i=1}^k I(R,n_i) \leq \ell\frac{\binom{c_0}{k}}{\binom{n}{k}} + (\ind(R) + \vep) \sum_{i=\ell+1}^k p_i^k.
    \end{equation}
    
    Suppose $\ell \geq 1$.  Using our bounds on $\vep$ and $n_0$ and (\ref{eqn:psumbound1}), we can further bound
    \begin{equation}\label{eqn:sumIconstant1}
        \frac{1}{\binom nk} \sum_{i=1}^k I(R,n_i) \leq \ell\frac{\binom{c_0}{k}}{\binom{n}{k}} + (\ind(R) + \vep) \sum_{i=\ell+1}^k p_i^k < 0.74\ind(R)
    \end{equation}
    and bound the product term
    \[ \frac{1}{\binom nk} \prod_{i=1}^k n_i \leq \frac{1}{\binom nk} c_0 n^{k-1} < \frac{2k!c_0}{n} < \vep. \]
    This yields $\ind(R,n) \leq 0.74\ind(R) + \vep < \ind(R)$, a contradiction.  Thus $\ell = 0$, so using (\ref{eqn:sumIbound1}) we may rewrite (\ref{eqn:I(R,n)case1_1}) as
    \begin{equation}\label{eq:i(R,n)case1_1}
        \ind(R,n) \leq (\ind(R) + \vep) \sum_{i=1}^k p_i^k + \frac{1}{\binom nk} \prod_{i=1}^k n_i.
    \end{equation}
    
    Isolating the product term and recalling the definition of $a$, as well as our lower bound on $n_0$,
    \[ \frac{1}{\binom nk} \prod_{i=1}^k n_i = \frac{n^k}{\binom nk} \prod_{i=1}^k p_i \leq (a+\vep)(k^k-k) \prod_{i=1}^k p_i. \]
    
    Plugging this into (\ref{eq:i(R,n)case1_1}) and recalling $\ind(R) = a + \gamma$,
    \[ \ind(R,n) \leq (a + \vep)\paren{\sum_{i=1}^k p_i^k + (k^k-k) \prod_{i=1}^k p_i} + \gamma \sum_{i=1}^k p_i^k \leq (a + \vep ) + 0.72\gamma. \]
    
    The first bound $\sum p_i^k + (k^k-k) \prod p_i\le 1$ is well-known (see, e.g.~(17) in~\cite{MR}) and the second bound comes from (\ref{eqn:psumbound1}).  This gives the contradiction
    \[ a + \gamma = \ind(R) \leq \ind(R,n) \leq a + 0.72\gamma + \vep  \]
    since $\vep < \gamma/100$.
\end{proof}

\subsection{The absurd case}

Now, we consider the other endpoint of $I$.

\begin{claim}
    If $\delta = 1 - \sum_i {\binom{n_i}{2}}/{\binom{n}{2}}$, then $\ind(R) \le a$.
\end{claim}

\begin{proof}
    If $\delta=1 - \sum_i {\binom{n_i}{2}}/{\binom{n}{2}}$, then (\ref{eqn:bound1}) implies that
    \begin{equation} \label{eqn:final1} I(R, n) \le \sum_{i=1}^k I(R, n_i) +\frac{2 \sum_{i \neq j} n_in_j}{k-2}\left(\frac{z}{k-2}\right)^{k-2}n^{k-2}.\end{equation}
    
    We first bound the second term.  Dividing by $\binom nk$ and again letting $p_i := n_i/n$, we reorganize
    \[ \frac{2}{\binom nk} \cdot \frac{\sum n_in_j}{k-2}\left(\frac{z}{k-2}\right)^{k-2}n^{k-2} =2 \cdot \frac{k^k-k}{(k-2)^{k-1}} \cdot \frac{n^k}{(n)_k} \cdot \paren{\sum_{i \ne j} p_i p_j} z^{k-2}a. \]
    
    Observe that $(k^{k-1} - 1)/(k - 2)^{k - 1}$ decreases to $e^2$.  In particular, for $k \geq 11$, we have $(k^k-k)/(k-2)^{k-1} \leq 7.5k$.  For $n>n_0$, we have $n^k/(n)_k < 1 + \vep$.  Finally, $\sum_{i \neq j} p_ip_j = (1-\sum p_i^2) / 2 \leq (1-1/k) / 2$  as $\sum p_i^2$ is minimized when $p_i = 1/k$ for all $i$.  Thus
    \[ \frac{2}{\binom nk} \cdot \frac{\sum n_in_j}{k-2}\left(\frac{z}{k-2}\right)^{k-2}n^{k-2} \leq 7.5(1+\vep) (k-1)z^{k-2}a < 0.25 a \] 
    for $k \geq 11$ as $(k-1)z^{k-2}$ is decreasing in $k$ and $(11-1)z^{11-2} < 10 \cdot 2^{-9} < 1/50$.  Using this and (\ref{eqn:sumIconstant1}) in (\ref{eqn:final1}), and $a<\ind(R)$ gives
    \[ \ind(R,n) \leq 0.74 \, \ind(R) + 0.25a < 0.99\, \ind(R).\]
    This  contradiction completes the proof of the claim and the theorem.
\end{proof}

\section{Proof of Theorem $\ref{thm:main2}$}\label{sec:rainbow_graph}

We give the proof of Theorem~\ref{thm:main2} in the following subsections.

\subsection{Setup}
Fix $k$ and $R = ([k],E)$ a rainbow colored graph with minimum degree at least $\eta (k-1)$ where $\eta > C \log k /(k-1)$. We may  assume that $k$ is sufficiently large by making $C$ sufficiently large so that the theorem is vacuous for small $k$. In particular, we will assume $k \geq 11$ so that we may use the same bounds as the previous section. It is notationally convenient to set $T = E \cup \{\emptyset\}$ and
view $R$ as a $T$-colored complete graph $([k], \binom{[k]}{2})$  with coloring function $\chi_R$ defined as follows:
\[ \chi_R (ij) = \begin{cases} \{i,j\} & ij \in E \\ \emptyset & ij \not\in E. \end{cases} \]

Our goal is to prove that $\ind(R) \le a$. To this end, fix $\gamma>0$ and assume for contradiction $\ind(R)=a+\gamma$. Next choose $\vep, c_0, n_0$ as in Section~\ref{sec:setup1}.

Suppose that $n>n_0$ is given and $H$ is a $T$-colored $n$-vertex graph with coloring function $\chi_H$ achieving $I(R,n)$. This implies 
$$I(R, H) = I(R, n) = \ind(R,n){\binom nk}$$
where $a+\gamma =\ind(R) \le \ind(R,n) \le \ind(R) + \varepsilon = a+\gamma + \varepsilon$.


Let $d_i (x)$, $d(x)$, $d(x,y)$, $d_{\{i, j\}} (x)$, $N_i (x)$, and $N_i^j (x)$
 be defined as in Section~\ref{sec:rainbow_clique}. Note that we do not have that all vertices in $N_j (x) \cap N_{j'}(x)$ for $j \neq j'$ have the same color to $x$ as it may be the case that $\chi_H (xy) = \emptyset$ and $\chi_H (xy') = \{j,j'\}$ for distinct $y, y' \in N_j (x) \cap N_{j'}(x)$.
We also do not have that $\cup_{j \ne i} N_i^j(x)$ is a partition of $N_i (x)$ as it may be the case that $y \in N_i^{j} (x) \cap N_i^{j'} (x)$ for some $j \ne j'$ satisfying $\chi_R (ij) = \chi_R (ij') = \emptyset$ and $y \in V(H)$ satisfying $\chi_H (xy) = \emptyset$. Thus we must develop new techniques to prove a version of (\ref{eqn:dn1}) from Section~\ref{sec:setup1} to obtain bounds on $d_i(x)$. This is the content of Section~\ref{sec:partition}.

As in Section~\ref{sec:setup1}, we partition $V(H)$ into $V_1 \cup \cdots \cup V_k$, $n_i=|V_i|$, where
$$V_i = \{x \in V(H):  |N_i(x)| \ge |N_j(x)| \hbox{ for all } j \ne i\}.
$$
If there is a tie, we break it arbitrarily.


\subsection{Partitioning argument}\label{sec:partition}

Let the {\em distance} between two vertices $v$ and $w$ in a graph $G$, denoted $\operatorname{dist}_G (v,w)$, be the number of edges in the shortest path between $v$ and $w$ in $G$. In our setting, a path cannot use an edge $e$ with $\chi(e) = \emptyset$. Then, define \[
\epsilon_G (v) := \max_{w\in V(G)} \operatorname{dist}_G (v,w),
\] the {\em eccentricity} of $v$ in $G$.  Note that the diameter $\operatorname{diam}(G) = \max_{v \in V(G)} \epsilon_G(v)$. For convenience, let $\epsilon (i) := \epsilon_R (i)$ for all $i\in [k]$.

Let $B (x)$ be the set of neighbors of $x$ in $H$. For $r \in \mathbb{N}$, let $k_r(i)$ be the number of vertices in $R$ at distance $r$ from $i$. Recall that $d_{\{i,j\}}(x)$ is the number of edges in $H$ incident with $x$ in color $\{i,j\}$

\begin{lemma}\label{lem:partition} 
    Let $i,j \in [k]$ with $\{i, j\} \in T$ and $x \in V(H)$. Then  \begin{align*}
        \text{\em (a)} ~~~ & d_i(x) \le \paren{\frac{|B (x)|}{k_1(i)}}^{k_1(i)} \paren{\frac{n - |B (x)|}{k - k_1(i) - 1}}^{k - k_1(i) - 1} \hspace{7 cm} \\
        \text{\em (b)} ~~~ & d_i(x) \le \paren{\frac{|N_i (x)|}{k-1}}^{k-1} \\
        \text{\em (c)} ~~~ & d_i(x) \le d_{\{i, j\}} (x) \cdot \paren{ \frac{n - d_{\{i, j\}} (x)}{k-2}}^{k - 2}. 
    \end{align*} Further, the number of copies of $R$ in $H$ containing vertices $x,y \in V(H)$ such that $x \in V(H)$ plays the role of vertex $i \in [k]$ in $R$ is at most \[ \paren{ \frac{|N_i (x)|}{k-2}}^{k-2}. 
    \]
\end{lemma}

\begin{proof}
    
    We start by proving the three upper bounds on $d_i (x)$. To count the number of copies of $R$ in $H$ where $x$ plays the role of $i$, we will recursively partition $|N_i (x)|$. First, we pick $k_1 (i)$ vertices from $B(x) \cap N_i (x) \subset V(H)$ to play the role of the vertices adjacent to $i$ in $R$. Notice that we may partition $B(x) \cap N_i (x)$ into $k_1 (i)$ parts based on the color of each vertex to $x$ as it uniquely determines its possible role in a copy of $R$.  Set $B_1 := B(x) \cap N_i (x)$. We now recursively define $B_r$ for all $r \in [\epsilon (i)]$. Let $2\le r \le \epsilon (i)$, let $m := k_1 (i) + \cdots + k_{r-1} (i)$, and suppose that we have chosen $y_1, y_2, \ldots, y_m \in V(H)$ to play the roles of all vertices at distance $r-1$ or less from $i$ in $R$. Then \begin{align*}
        B_r := & \, B_r (x, B_1, B_2, \ldots, B_{r-1}, y_1,\ldots, y_{k_{r-1} (i)}) \\
        = & \,  N_i (x) \cap \paren{B (y_1) \cup \cdots \cup B (y_{k_{r-1} (i)})} \setminus \paren{x \cup B_1 \cup \cdots \cup B_{r-1}}.
    \end{align*} Here, $B_r$ is the set of  vertices in $H$ that can play the role of vertices at distance $r$ from $i$ in $R$, given that we have already selected all vertices at distance at most $r-1$ from $i$.
    
    Note that by definition,  $B_r \cap B_\ell = \emptyset$ for all $r, \ell \in [\epsilon (i)]$ and $\bigcup B_r \subseteq N_i (x)$. For the remainder of the proof, we write $k_r := k_r (i)$ for all $r \in [\epsilon (i)]$ for convenience.
    
    Each vertex $v \in B_r$ has an edge to at least one of $y_1, \ldots, y_{k_{r-1}}$. The color of this edge uniquely determines the role that $v$ may play in a copy of $R$, so this allows us to uniquely partition $B_{r}$ into $k_r$ parts. We note that it may be the case that $v$ cannot legally play any role, but that only decreases the number of possible copies of $R$, so we may assume that this does not occur. Let $P_r := P (B_r, k_r)$ be the set of tuples $\vec{y} \in B_r^{k_r}$ with one vertex from each part of $B_r$,
     so
     $$|P_r| \le p(|B_r|, k_r).$$This gives 
$$        d_i (x) \le  \sum_{\vec{y}_1 \in P_1}\cdots \sum_{\vec{y}_{\epsilon (i) - 1} \in P_{\epsilon (i) - 1}} p(|B_{\epsilon (i)}|, k_{\epsilon (i)}) 
        \le \sum_{\vec{y}_1 \in P_1}\cdots \sum_{\vec{y}_{\epsilon (i) - 1} \in P_{\epsilon (i) - 1}} p\paren{|N_i (x)| - \sum_{r=1}^{\epsilon (i) - 1} |B_r|, k_{\epsilon (i)}}.$$ Using (\ref{eqn:pq}) from Section~\ref{sec:setup1} we see that
        \begin{align*}
          \sum_{\vec{y}_{\epsilon (i) - 1} \in P_{\epsilon (i) - 1}} ~ p\paren{|N_i (x)| - \sum_{r=1}^{\epsilon (i) - 1} |B_r|, k_{\epsilon (i)}} 
           &= p \paren{|B_{\epsilon (i) - 1}|, k_{\epsilon (i) - 1}} \cdot p\paren{|N_i (x)| - \sum_{r=1}^{\epsilon (i) - 1} |B_r| , k_{\epsilon (i)}} \\
           &=p\paren{|N_i (x)| - \sum_{r=1}^{\epsilon (i) - 2} |B_r| , k_{\epsilon (i) - 1} (i) + k_{\epsilon (i)}}. 
        \end{align*}
        Using $\sum_{r = 1}^{\epsilon (i)} k_r = k-1$, we obtain
\begin{align*}d_i(x) &\le \sum_{\vec{y}_1 \in P_1}\cdots \sum_{\vec{y}_{\epsilon (i) - 2} \in P_{\epsilon (i) - 2}} ~ p\paren{|N_i (x)| - \sum_{r=1}^{\epsilon (i) - 2} |B_r| , k_{\epsilon (i) - 1} + k_{\epsilon (i)}}\\
&= \sum_{\vec{y}_1 \in P_1}\cdots \sum_{\vec{y}_{\epsilon (i) - 2} \in P_{\epsilon (i) - 2}} ~ p\paren{|N_i (x)| - \sum_{r=1}^{\epsilon (i) - 2} |B_r| , (k-1)-\sum_{r=1}^{\epsilon(i)-2} k_r}.
\end{align*}
        
        Continuing this process, we obtain, for each $1 \le \ell \le \epsilon(i)-1$,
  $$  d_i(x) 
  \le \sum_{\vec{y}_1 \in P_1}\cdots \sum_{\vec{y}_{\ell} \in P_\ell} p\paren{|N_i (x)| - \sum_{r=1}^{\ell} |B_r| , (k-1) - \sum_{r=1}^{\ell} k_r}. $$
  When $\ell=1$ this becomes
  \begin{align*}
        d_i (x)
        &\le \sum_{\vec{y}_1 \in P_1} p\paren{|N_i (x)| - |B_1|, k -1- k_1} \addtocounter{equation}{1}\tag{\theequation} \label{eqn:partition_P1} \\
        &= p(|B (x)|, k_1) \cdot p\paren{|N_i (x)| - |B (x)|, k - k_1 - 1}.
    \end{align*} As $|N_i (x)| \le n-1 < n$ for all $i \in [k]$,  \[
    p(|B (x)|, k_1) \cdot p\paren{|N_i (x)| - |B (x)|, k - k_1 - 1} \le \paren{\frac{|B (x)|}{k_1}}^{k_1} \paren{\frac{n - |B (x)|}{k - k_1 - 1}}^{k - k_1 - 1},
    \] so  (a) holds. Alternatively, (\ref{eqn:pq}) also yields \[
    p(|B (x)|, k_1) \cdot p\paren{|N_i (x)| - |B (x)|, k - k_1 - 1} \leq p(|N_i(x)|, k-1) \leq \paren{\frac{|N_i (x)|}{k-1}}^{k-1}, \]
     so (b) holds.  

    For (c), let $j \in [k]$ such that $ij \in E$. We bound $d_i (x)$ as before, but we choose the vertex $y$ that plays role $j$ separately. We see that \[
    |P_1| \le d_{\{i, j\}} (x) \cdot p (|B_1| - d_{\{i,j\}}(x), k_1 -1).
    \] This combined with (\ref{eqn:partition_P1}) and (\ref{eqn:pq}) gives \begin{align*}
        d_i (x) &\le \sum_{\vec{y}_1 \in P_1} p\paren{|N_i (x)| - |B_1|, k - k_1 - 1} \\
        &\le d_{\{i, j\}} (x) \cdot p (|B_1| - d_{\{i, j\}} (x), k_1 -1) \cdot p\paren{|N_i (x)| - |B_1|, k - k_1 - 1} \\
        &\le d_{\{i, j\}} (x) \cdot p (n - d_{\{i, j\}} (x), k -2) \\
        &\le d_{\{i, j\}} (x) \cdot \paren{ \frac{n - d_{\{i, j\}} (x)}{k-2}}^{k - 2}. 
    \end{align*}

    It remains to prove the last sentence of the lemma. We proceed as before except that, for $\ell \in [\epsilon (i)]$ such that $y \in B_\ell$, we require that $y$ is chosen. This means that instead of choosing $k_{\ell} (i)$ vertices from $B_{\ell}$, we only need to choose $k_{\ell} (i)-1$ vertices from $B_{\ell}$ as we have already chosen $y$. Following the same procedure as  before, we see that \begin{align*}
        d_i (x)& \le 
         \sum_{\vec{y}_1 \in P_1}\cdots \sum_{\vec{y}_{\ell} \in P_{\ell}} p \paren{|B_{\ell}|, k_{\ell} - 1} \cdot p\paren{|N_i (x)| - \sum_{r=1}^{\ell + 1} |B_r| , (k-1) - \sum_{r=1}^{\ell + 1} k_r} \\
        &\le \sum_{\vec{y}_1 \in P_1}\cdots \sum_{\vec{y}_{\ell} \in P_\ell} p\paren{|N_i (x)| - \sum_{r=1}^{\ell} |B_r| , (k-2) - \sum_{r=1}^{\ell} k_r} \\
        &\vdots \\
        &\le \sum_{\vec{y}_1 \in P_1} p\paren{|N_i (x)| - |B_1|, k - 2-k_1} \\
        &= p(|B (x)|, k_1 \cdot p\paren{|N_i (x)| - |B (x)|, k - 2-k_1} \\
        &\le p(|N_i (x)|, k - 2) \\
        &\le \paren{ \frac{|N_i (x)|}{k-2}}^{k-2}.
    \end{align*} This completes the proof.
\end{proof}


\subsection{Minimum degree}

As in Section~\ref{sec:mindeg1}, we wish to show that each vertex of $H$ lies in approximately the average number of copies of $R$.

\begin{lemma}\label{lem:min_deg2}
    $d(x) \geq a n^{k-1} / (k-1)!$ for all $x \in V(H)$.
\end{lemma}

This follows from an identical Zykov symmetrization argument as used in the proof of Lemma~\ref{lem:min_deg1}.  Note that we have assumed the same inequalities for $n_0$ as we did in Section~\ref{sec:setup1}.

\subsection{Maximum color and non-edge degrees}

The following two claims are used in the proof of Lemma~\ref{lem:nbhd_z2} to bound the size of the second largest neighborhood.

Let $$\alpha := \frac{\max_{x,i,j} d_{\{i,j\}}(x)}{n}$$ where the maximum is taken over all vertices $x \in V(H)$ and all colors $\{i,j\} \in T$. We upper bound this value.

\begin{claim}\label{claim:max_color_deg2}
    $\alpha < \eta/4$.
\end{claim}

\begin{proof}
    Let $x$ achieve this maximum, so that $d_{\{i,j\}}(x) = \alpha n$ for some $\{i,j\} \in T$.
    By Lemma~\ref{lem:partition}(c) and $\alpha \leq 1$, we get
    \begin{align*}
        \max\{d_i(x), d_j(x)\} \leq \alpha n \paren{\frac{(1-\alpha)n}{k-2}}^{k-2} \leq n^{k-1} \paren{\frac{1-\alpha}{k-2}}^{k-2}.
    \end{align*}

    For any other vertex $\ell \ne i,j$, we have $|N_\ell(x)| \leq (1-\alpha)n$, since $\ell$ is adjacent to no edges of color $\{i,j\}$ in $R$.  Thus by Lemma~\ref{lem:partition}(b), we get
    \[ d_\ell(x) \leq \paren{\frac{|N_\ell (x)|}{k-1}}^{k-1} \leq \paren{\frac{(1-\alpha)n}{k-1}}^{k-1} \leq n^{k-1} \paren{\frac{1-\alpha}{k-2}}^{k-2}. \]
    The last inequality comes as decreasing the denominator increases the fraction, and the base is less than 1, so decreasing the exponent increases the result.  Summing over all indices in $[k]$ and using Lemma~\ref{lem:min_deg2}, we get
    \begin{equation*}
         \frac{1}{k^{k-1}-1} = \frac{a}{(k-1)!} \leq \frac{d(x)}{n^{k-1}} \leq k \paren{\frac{1-\alpha}{k-2}}^{k-2}.
    \end{equation*} Rearranging yields \[
    \frac{1}{k} \cdot \frac{(k-2)^{k-2}}{k^{k-1}-1} \le (1-\alpha)^{k-2}.
    \]

    We see that
    \[ \frac{(k-2)^{k-2}}{k^{k-1}-1} \geq \frac{(k-2)^{k-2}}{k^{k-1}} = \frac1k \paren{1-\frac2k}^{k-2} > \frac1{e^2k}, \]
    so
    \[ \frac{1}{e^2k^2} < (1-\alpha)^{k-2} \leq \exp(-(k-2)\alpha). \]

    Assume for contradiction that $\alpha \geq \eta/4 > C\log k/(4(k-1))$.  Then \[
    \frac{1}{e^2k^2} < k^{-C(1-1/(k-1))/4} < k^{-0.9C/4},
    \]  since $k \geq 11$.  For $C > 10$, this gives a contradiction for sufficiently large $k$.
\end{proof}

Let $$\beta := \frac{\max_{x} d_\emptyset(x)}{n}$$ where the maximum is taken over all vertices $x \in V(H)$ and $d_\emptyset(x)$ is the number of edges in $H$ incident with $x$ in color $\emptyset$ (non-edges). Note that we may assume that $R$ has at least one non-edge, since otherwise the proof from Section~\ref{sec:rainbow_clique} suffices. Thus we may also assume that $H$ has at least one non-edge, so $\beta >0$. We upper bound $\beta$.

\begin{claim}\label{claim:max_nonedge_deg2}
    $\beta < 1-\eta/2$.
\end{claim}

\begin{proof}
    Assume for contradiction that $\beta \geq 1 - \eta/2$. Let $x$ achieve this maximum so that $d_\emptyset (x) = \beta n$. This implies that $B(x) = (1 - \beta) n$. For any $i \in V(R)$, Lemma~\ref{lem:partition}(a) gives 
    \begin{align*}
        d_i (x) &\le \paren{\frac{B (x)}{k_1(i)}}^{k_1(i)} \paren{\frac{n - B (x)}{k - k_1(i) - 1}}^{k - k_1(i) - 1} \\
        &= \paren{\frac{(1- \beta) n}{k_1(i)}}^{k_1(i)} \paren{\frac{\beta n}{k - k_1(i) - 1}}^{k - k_1(i) - 1} \\
        &= n^{k-1} \frac1{k_1 (i)^{k_1 (i)}} (1-\beta)^{k_1 (i)}\beta^{k- k_1 (i) -1} \paren{\frac{1}{k-1-k_1 (i)}}^{k-1-k_1 (i)}.
    \end{align*}  
    For this section, we take the convention $0^0 = 1$ to handle the case that $k_1(i) = k-1$.
    
    Let $q = k_1(i)/(k-1) \in (0,1]$.  Then
    \begin{equation}\label{eqn:beta_q_final}
        d_i (x) \le \paren{\frac{n}{k-1}}^{k-1}  \paren{\frac{(1-\beta)^q\beta^{1-q}}{q^q(1-q)^{1-q}}}^{k-1}.
    \end{equation}

    We will first bound the term
    \begin{equation}\label{eqn:beta_q_1}
        \frac{(1-\beta)^q \beta^{1-q}}{q^q (1-q)^{1-q}}.
    \end{equation}

    Regarding (\ref{eqn:beta_q_1}) as a function of $\beta$, we see that the derivative
    \[ \frac{\partial}{\partial \beta}\paren{\frac{(1-\beta)^q \beta^{1-q}}{q^q (1-q)^{1-q}}} = \frac{(1-\beta)^{q-1}\beta^{-q}}{q^q(1-q)^{1-q}}((1-q)-\beta) \]
    is negative for $\beta > 1-q$ since the fraction is nonnegative.  Recall that $k_1(i) = \deg_R(i) \geq \eta (k-1)$ by assumption, so $q > \eta$.  We have also assumed for contradiction that $\beta \geq 1-\eta/2 > 1-\eta > 1-q$. Thus decreasing $\beta$ to $1 - \eta/2$ will only increase (\ref{eqn:beta_q_1}), i.e.
    \begin{equation}\label{eqn:beta_q_2}
        \frac{(1-\beta)^q \beta^{1-q}}{q^q (1-q)^{1-q}} \leq \frac{(\eta/2)^q (1-\eta/2)^{1-q}}{q^q (1-q)^{1-q}}.
    \end{equation}

    We now have a function purely of $q$.  Taking the derivative, we get \[ \frac{\partial}{\partial q} \paren {\frac{(\eta/2)^q (1-\eta/2)^{1-q}}{q^q (1-q)^{1-q}}} = \frac{1}{2} (1 - q)^{q - 1} q^{-q} (2 - \eta)^{1-q} \eta^q \log \paren{\frac{\eta (1-q)}{q (2 - \eta)}} \] where all terms are positive except the logarithm, which is negative for $q > \eta/2$. Thus (\ref{eqn:beta_q_2}) is decreasing with $q$ for $q > \eta/2$, so we may take the further upper bound 
    \begin{equation}
        \frac{(1-\beta)^q \beta^{1-q}}{q^q (1-q)^{1-q}} \leq \frac{(\eta/2)^\eta (1-\eta/2)^{1-\eta}}{\eta^\eta (1-\eta)^{1-\eta}} = 2^{-\eta} \paren{1+\frac{\eta}{2(1-\eta)}}^{1-\eta} \leq \exp(-(\log 2 - 1/2)\eta).
    \end{equation}

    We now have an appropriate upper bound.  Substituting into (\ref{eqn:beta_q_final}) and recalling that $\eta > C \log k / k$, we see that
    \[ d_i(x) \leq \paren{\frac{n}{k-1}}^{k-1} \exp(-(\log 2 - 1/2)(k-1)\eta) < \paren{\frac{n}{k-1}}^{k-1} k^{-C(\log 2-1/2)}. \]

    Using Lemma~\ref{lem:min_deg2}, we get 
    \[ \frac{1}{k^{k-1}-1} = \frac{a}{(k-1)!} \leq \frac{d(x)}{n^{k-1}} \leq k \paren{\frac{1}{k-1}}^{k-1} k^{-C(\log 2-1/2)}. \]

    Rearranging terms and using the standard inequality (\ref{eq:e_inverse1}) yields
    \[ \frac{1}{ek} \leq k^{-C(\log 2-1/2)}. \]

    For $C > 1/(\log 2-1/2) \approx 5.18$, this yields a contradiction for sufficiently large $k$.
\end{proof}

\subsection{The second largest neighborhood}

For a vertex $x \in V(H)$, let $Z(x)$ be the second largest set 
in  $\{N_1(x), \ldots, N_k(x)\}$ and define 
$$z:= z_{k,n} = \max_{x \in V(H)} \frac{|Z(x)|
}{n}.$$

\begin{lemma}\label{lem:nbhd_z2}
    $z < 1-\eta/8$.
\end{lemma}

\begin{proof}
    Let $x \in V(H)$ such that $z=|Z(x)|/n$.  Suppose $x \in V_i$ and $Z(x) = N_j(x)$ for distinct $i,j \in [k]$.  Then we want to bound $|N_i(x) \cap Z(x)|$.
    Suppose $y \in N_i(x) \cap Z(x)$.  If $xy \in E$, then $i \in \chi_H(xy)$ and $j \in \chi_H(xy)$, so $\chi_H(xy) = \{i,j\}$.  Thus $|N_1(x)\cap Z(x)| \leq d_{\{i,j\}}(x) + d_\emptyset(x)$.
    It follows that
    $$|N_i(x)|+|Z(x)|
    =|N_i(x) \cup Z(x)|+|N_i(x) \cap Z(x)| \le n + d_{\{i,j\}} (x) + d_\emptyset(x) \leq (1+\alpha+\beta) n.$$ Thus $|N_i(x)| + |Z(x)| < (2-\eta/4)n$ by Claims~\ref{claim:max_color_deg2} and~\ref{claim:max_nonedge_deg2}.  Since $|Z(x)| \leq |N_i(x)|$, this gives $z < 1-\eta/8$.
\end{proof}

\subsection{One large part}

We now take care of the situation when one of the $V_i$'s is very large.

\begin{lemma}\label{lem:large_V2}
    $|V_i| \le (1-1/3k)n$ for all $i \in [k]$.
\end{lemma}

\begin{proof}
    By contradiction, WLOG suppose that $|V_1|>(1-1/3k)n$. If $x \in V_1$, then $|N_1(x)|\ge |N_i(x)|$ for all $i>1$ so
    $|N_2(x)| \le |Z(x)| \le z n$.  Applying Lemma~\ref{lem:partition}(b) to $d_2 (x)$ gives 
    $$a{\binom{n}{k}} \le I(R, H) = \sum_{x \in V(H)} d_2(x) \le |V_1| \paren{\frac{zn}{k-1}}^{k-1} + \frac{n}{3k} \paren{\frac{n}{k-1}}^{k-1} \leq \left(z^{k-1}+\frac{1}{3k}\right) \frac{n^k}{(k-1)^{k-1}}.$$ 
    Rearranging and using $n_0^k/(n_0)_k < 1.01$ as assumed in (\ref{eq:nk_nk}), we get
    \begin{equation*}
        \frac{(k-1)^{k-1}}{k^k-k} \leq 1.01 \paren{z^{k-1} + \frac{1}{3k}}.
    \end{equation*}
    Using the standard inequality (\ref{eq:e_inverse1}) and then Lemma~\ref{lem:nbhd_z2} gives
    \[ \paren{\frac1{1.01e} - \frac13}\frac1k <  z^{k-1} < \paren{1-\frac\eta8}^{k-1} \leq \exp( -C\log k/8 ) = k^{-C/8}. \]

    For any $C>8$ this fails to hold for sufficiently large $k$.
    \end{proof}

\subsection{Counting the copies of $R$ in $H$}\label{sec:counting_copies}

The way we count copies of $R$ in $H$ is very similar to the previous section and to~\cite{MR}.  While we do not have as much information in this case, without a focus on optimizing for small $k$, we allow ourselves to be less strict with the counting arguments.

Call a copy $f$ of $R$ in $H$ {\em transversal} if it includes exactly one vertex in $V_i$ for all $i \in [k]$. We partition the copies of $R$ in  $H$  as $H_m \cup H_g \cup H_b$ where $H_m$ comprises those copies that lie entirely inside some $V_i$, $H_g$ comprises those copies that intersect every $V_i$ whose edge coloring coincides with the natural one given by the vertex partition (meaning the map from $R$ to $H$ takes vertex $i$ to a vertex in $V_i$), and $H_b$ comprises all other copies of $R$ (these include transversal copies, but some vertex in any such copy will be in an inappropriate $V_i$).  Thus a transversal copy $f$ is in $H_b$ if and only if the unique map $\phi : [k] \to f$ with $\phi(i) \in V_i$ for all $i$ is not a graph isomorphism from $R \to H[f]$.  Let $h_m=|H_m|, h_g=|H_g|$ and $h_b=|H_b|$ so that
$$I(R, H) = h_m+h_g+h_b.$$

We will bound each of these three terms separately.  As in Section~\ref{sec:counting1}, let $D$ be the set of all pairs $\{v,w\}$ such that $v \in V_i, w \in V_j,$ and $\chi_H(vw) \ne \chi_R(ij)$ where $i \ne j$.  Let $\delta := |D|/\binom n2$.  The identical reasoning as in Section~\ref{sec:counting1} gives the first two bounds
\begin{equation} \label{eqn:hmbound2}
	h_m = \sum_{j=1}^k I(R, H[V_j]) \le \sum_{j=1}^k  I(R, n_j)
\end{equation}
and

\begin{equation} \label{eqn:delta+2}
	h_g \le\prod_{\ell=1}^k n_{\ell}\left(1-\frac{\delta{\binom{n}{2}}}{\sum_{1\le i<j\le k}n_in_j}\right) = \prod_{\ell=1}^k n_{\ell}\left(1-\frac{\delta{\binom{n}{2}}}{{\binom{n}{2}} - \sum_i {\binom{n_i}{2}}}\right).
\end{equation}

Our next task is to upper bound $h_b$.  This argument must be carried out differently.  For a vertex $x \in V(H)$ and $j \in [k]$, recall that $N_j(x) \subset V(H)$ is the set of  $y$ such that $x,y$ lie in a copy of $R$ with $x$ playing the role of vertex $j$ in $R$.
Let us enumerate the set $J$ of ordered pairs $(e,f)$ where $e\in D, f \in H_b$, and $e \subset f$.



We must show that each $f \in H_b$ contains an edge in $D$.  If $f$ is transversal, then as we have noted, the natural map is not a graph isomorphism. Thus there is some incorrectly colored edge which is in $D$.  If $f$ is not transversal, there is some $i \in [k]$ such that $|f \cap V_i|\ge 2$. Note that $f \notin H_m$, so $|f \cap V_i| < k$. As $R$ is connected, there exist $v \in V_i,~u \in V_j$ for some $j \ne i$ such that $vu$ is an edge in $f$. Since $|f \cap V_i| \geq 2$, choose also $w \in f \cap V_i$ with $w \ne v$.
If $\chi_R (ij) = \emptyset$ then $vu \in D$. If $\chi_R (ij) = \{ i, j\},$ then as $\chi_H (vu) = \chi_H (wu) = \{ i, j\}$ would contradict that $f$ is a copy of $R$ in $H$, we must have that $uv$ or $uw$ in $D$.
This gives us that \[ h_b \le |J| .\]

To bound $|J|$ from above, we start by choosing some bad edge $vw \in D$.  Let $f \subset V(H)$ such that $(vw,f) \in J$.  Either $v \in V_i$ does not play the role of $i$ or $w \in V_j$ does not play the role of $j$ in $f$.  Then $f \subset N_\ell (v) \cup \{ v \}$ for some $\ell \neq i$ or $f \subset N_\ell(w) \cup \{w\}$ for some $\ell \neq j$.  We have $|N_\ell(v)|,|N_\ell(w)| \leq zn$ by the definition of $z$ and the partition $V_1 \cup \cdots \cup V_k = V(H)$.  By the final statement of Lemma~\ref{lem:partition}, 
\[ |J| 
\le \sum_{vw \in D} \paren{\sum_{\ell \ne i} \paren{\frac{|N_\ell(v)|}{k-2}}^{k-2} + \sum_{\ell \ne j} \paren{\frac{|N_\ell(w)|}{k-2}}^{k-2}} \leq 2|D| (k-1) \paren{\frac{zn}{k-2}}^{k-2} . \] 

Thus, recalling that $\delta := |D|/\binom n2$, we obtain \begin{equation}\label{eqn:delta-2} h_b \leq 2\delta (k-1) \binom{n}{2} \paren{\frac{zn}{k-2}}^{k-2}.
\end{equation}

Using (\ref{eqn:hmbound2}), (\ref{eqn:delta+2}), and (\ref{eqn:delta-2})  we obtain\begin{equation}\label{eqn:bound2}
    I(R, n) \leq \sum_{i} I(R, n_i) +
    \prod_{\ell} n_{\ell}\left(1-\frac{\delta{\binom{n}{2}}}{{\binom{n}{2}} - \sum_i {\binom{n_i}{2}}}\right)
    + 2\delta (k-1) \binom{n}{2} \paren{\frac{zn}{k-2}}^{k-2}.
\end{equation}
Our final task is to upper bound the RHS.

As in Section~\ref{sec:counting_copies}, we see that $\delta \in I\df \left[0, 1 - \sum_i {\binom{n_i}{2}}/{\binom{n}{2}} \right]$. Viewing (\ref{eqn:bound2}) as a linear function of $\delta$, it again suffices to check the endpoints of $I$.

\subsection{The extremal case}\label{sec:extremal2}

\begin{claim}
    If $\delta = 0$, then $\ind(R) \le a$.
\end{claim}

\begin{proof}
    If $\delta=0$, then (\ref{eqn:bound2}) implies that
    \begin{equation*}
    I(R, n) \le \sum_{i=1}^k I(R, n_i) + \prod_{i=1}^k n_i.\end{equation*}

    This is the same equation as (\ref{eqn:I(R,n)case1_1}), and we have all the same assumptions.  The same argument as in Section~\ref{sec:extremal1} derives a contradiction.
\end{proof}

\subsection{The absurd case}

Now, we consider the other endpoint of $I$.

\begin{claim}
    If $\delta = 1 - \sum_i {\binom{n_i}{2}}/{\binom{n}{2}}$, then $\ind(R) \le a$.
\end{claim}

\begin{proof}
    If $\delta=1 - \sum_i {\binom{n_i}{2}}/{\binom{n}{2}}$, then (\ref{eqn:bound2}) implies that
    \begin{equation} \label{eqn:final2} I(R, n) \le \sum_{i=1}^k I(R, n_i) + 2(k-1) \sum_{i \neq j} n_in_j \left(\frac{z}{k-2}\right)^{k-2}n^{k-2}.\end{equation}
   This is similar to (\ref{eqn:final1}) with an extra factor of approximately $k^2$ in the second term.  We can bound the first sum using the same techniques as in Section~\ref{sec:extremal1}, giving (\ref{eqn:sumIconstant1}):

    \begin{equation}\label{eqn:sumIconstant2}
        \frac{1}{\binom nk} \sum_{i=1}^k I(R,n_i) \leq \ell\frac{\binom{c_0}{k}}{\binom{n}{k}} + (\ind(R) + \vep) \sum_{i=\ell+1}^k p_i^k < 0.74 \ind(R).
    \end{equation}
    
    We now bound the second term.  Dividing by $\binom nk$, we reorganize
    \[ \frac{2}{\binom nk} (k-1) \paren{\sum_{i \ne j} n_in_j}\left(\frac{z}{k-2}\right)^{k-2}n^{k-2} = 2 \cdot \frac{(k-1)(k^k-k)}{(k-2)^{k-2}} \cdot \frac{n^k}{(n)_k} \cdot \paren{\sum_{i \ne j} p_i p_j} z^{k-2}a. \]

    We first relax $(k-1)(k^k-k) < k^{k+1}$.  Observe that $k^{k-2}/(k-2)^{k-2} \leq e^2$.  Thus this first quotient is at most $e^2k^3$.  For $n>n_0$, we have $n^k/(n)_k < 1 + \vep \leq 1.01$.  Finally, $\sum_{i \neq j} p_ip_j = (1-\sum p_i^2) / 2 \leq (1-1/k) / 2 \leq 1/2$  as $\sum p_i^2$ is minimized when $p_i = 1/k$ for all $i$.  Thus
    \[ \frac{2}{\binom nk} (k-1) \paren{\sum_{i \ne j} n_in_j}\left(\frac{z}{k-2}\right)^{k-2}n^{k-2} \leq 2 \cdot e^2 k^3 \cdot 1.01 \cdot \frac{1}{2} \cdot z^{k-2}a = 1.01e^2k^3 z^{k-2} a. \]

    By Lemma~\ref{lem:nbhd_z2}, we have
    \[ z^{k-2} \leq (1-\eta/8)^{k-2} \leq \exp\paren{-\frac{(k-2)\eta}8} \leq \exp\paren{-\frac C8 \paren{1 - \frac{1}{k-1}} \log k} < k^{-0.9C/8}. \]

    We again used $k \geq 11$ here.  Thus for $C > 24/0.9 \approx 26.67$, we have $1.01e^2k^3z^{k-2}a < 0.25a$ for large enough $k$.
    Recalling that $a<\ind(R)$, plugging this and (\ref{eqn:sumIconstant2}) into (\ref{eqn:final2}) gives
    \[ \ind(R,n) \leq 0.74 \ind(R) + 0.25 \, a < 0.99\ind(R).\]
    This  contradiction completes the proof of the claim and the theorem.
\end{proof}

\section{Disconnected rainbow graphs}\label{sec:matching}













In this section, we show that rainbow graphs with multiple connected components are not fractalizers.

Let $R = (V,E)$ be a rainbow graph with $k$ vertices and $\ell > 1$ connected components.  Let $R = R_1 \cup \cdots \cup R_\ell$ be the connected components of size $c_1,\ldots,c_\ell$ respectively.  Assume also $c_i \geq 2$ for all $i$ (no isolated vertices).  We will show that $R$ is not a fractalizer.

We begin by upper bounding the number of copies $I(R,G_n)$ for $G_n \in \mathcal{G}_R(n)$ an iterated balanced blowup.  Then for any $i \in [\ell]$, by the same argument as for computing the inducibility of the iterated balanced blowup (see e.g.~\cite{PG}),
\[ I(R_i,G_n) = \paren{\frac nk}^{c_i} + k\paren{ \frac{n}{k^2}}^{c_i} + k^2\paren{\frac{n}{k^3}}^{c_i} + \cdots = (1+o(1)) \frac{n^{c_i}}{k^{c_i}-k}.  \]

Any $S \subset V(G_n)$ with $G_n[S] \cong R$ has a unique partition $S = S_1 \cup \cdots \cup S_\ell$ where $G_n[S_i] \cong R_i$.  Thus we can upper bound
\begin{equation}\label{eqn:ind_oneblowup}
    I(R,G_n) \leq \prod_{i=1}^\ell I(R_i,G_n) = (1+o(1))n^k \prod_{i=1}^\ell \frac{1}{k^{c_i}-k}.
\end{equation}

However, consider instead the family of graphs $\mathcal{H}(n)$ consisting of separate iterated balanced blowups of each part.  Formally, $H\in \mathcal{H}(n)$ if $|V(H)| = n$ and we have a partition $V(H) = V_1 \cup \cdots \cup V_\ell$ with the following properties:
\begin{enumerate}
    \item For all $i \in [\ell]$, $\big| |V_i| - \frac{c_i}{k} n \big| \leq 1$.
    \item For all $i \in [\ell]$, the induced subgraph $G[V_i] \in \mathcal{G}_{R_i}(|V_i|)$.
    \item For all $v \in V_i, w \in V_j$ with $i \ne j$, we have $vw \notin E(H)$.
\end{enumerate}

In $\mathcal{H}(n)$, there are no edges between any copy of $R_i$ and any copy of $R_j$ for distinct $i,j$.  Since $R$ is rainbow, copies of each component $R_i$ exist only in $V_i$.  Then for $H_n \in \mathcal{H}(n)$, we have
\begin{equation}\label{eqn:ind_manyblowup}
    I(R,H_n) = \prod_{i=1}^\ell I(R_i,H_n[V_i]) = \prod_{i=1}^\ell (1+o(1)) \frac{c_i!}{c_i^{c_i}-c_i} \binom{\frac{c_i}{k}n}{c_i} = (1+o(1))n^k \prod_{i=1}^\ell \frac{1}{k^{c_i}-k(\frac{k}{c_i})^{c_i-1}}.
\end{equation}

Comparing (\ref{eqn:ind_oneblowup}) with (\ref{eqn:ind_manyblowup}), we subtract larger numbers in the denominator of (\ref{eqn:ind_manyblowup}), so the family of graphs $\mathcal{H}(n)$ induces asymptotically more copies than the family $\mathcal{G}_R(n)$.  Thus $R$ is not a fractalizer.  Since $R$ was generic, disconnected rainbow graphs without isolated vertices are not fractalizers.







\section{Appendix}

\begin{proof}[Proof of (\ref{eqn:pq})]
    Let $q, q' \ge 0$ and $t, t'>0$. Recall that $p(q,t)$ is the maximum of $\prod_i q_i$ where $q_1 + \cdots + q_t = q$ and each $q_i \ge 0$ is an integer. Let $q_1, \ldots q_t$ integers such that $p(q, t) = \prod_{i = 1}^t q_i$ and $q_1', \ldots q_{t'}'$ integers such that $p(q', t') = \prod_{i = 1}^{t'} q_i'$. Then, \[
    q_1 + \cdots + q_t + q_1'+ \cdots + q_{t'}' = q + q'.
    \] Thus, the fact that $p(q +q', t + t')$ is a maximum gives that \[
    p(q +q', t + t') \ge \prod_{i = 1}^t q_i \prod_{i = 1}^{t'} q_i' = p(q, t) p(q', t')
    \] as desired.
\end{proof}

\begin{proof}[Proof of (\ref{eqn:max1})]
    We will show that \[\frac{1}{k^{k-1}-1} > 1.4 \left(\frac{0.6}{k-2}\right)^{k-2}\] for all $k \geq 11$. This is true for $k=11$. By  (\ref{eq:e_inverse1}) and the fact that $k \geq 11$, \begin{align*}
        \frac{1}{k^{k-1}-1} &\geq \frac{1}{e (k-1)^{k-1}}.
    \end{align*} We will prove that \[\frac{1}{e (k-1)^{k-1}} > 1.4 \left(\frac{0.6}{k-2}\right)^{k-2}\] for $k \ge 12$ by induction on $k$. For $k = 12$, plugging in certifies that this is true. By the inductive hypothesis, assume that \begin{equation} \label{eq:cd}
        \frac{1}{e (k-2)^{k-2}} > 1.4 \left(\frac{0.6}{k-3}\right)^{k-3}.
    \end{equation} We see that \begin{align} \label{eq:ab}
        &\frac{(k-2)^{k-2}}{(k-1)^{k-1}} > \frac{0.6 (k-3)^{k-3}}{(k-2)^{k-2}}, \text{ or equivalently } f(k) := \frac{(k-2)^{2k-4}}{(k-1)^{k-1} (k-3)^{k-3}} > 0.6
    \end{align} since \[ f(11) \approx 0.89 > 0.6 \] and \begin{align*} 
        \frac{d}{dk} f(k) &= - \frac{(k-2)^{2k-4}}{(k-1)^{k-1}(k-3)^{k-3}} \cdot\ln \paren{ \frac{(k-1) (k-3)}{(k-2)^2} } \geq 0
    \end{align*} since $(k-2)^2 > (k-1) (k-3)$. Then, by (\ref{eq:cd}) and (\ref{eq:ab}), we see that \begin{align*}
         \frac{1}{e (k-1)^{k-1}} &= \frac{1}{e (k-2)^{k-2}} \cdot \frac{(k-2)^{k-2}}{(k-1)^{k-1}} \\[.2 cm]
         &> 1.4 \left(\frac{0.6}{k-3}\right)^{k-3} \cdot \frac{0.6 (k-3)^{k-3}}{(k-2)^{k-2}} = 1.4 \left(\frac{0.6}{k-2}\right)^{k-2}.
    \end{align*}
\end{proof}

\begin{proof}[Proof of (\ref{eqn:large1})]
    We will show that \[ \frac{1}{1 + \vep} \cdot \frac{(k-1)^{k-1}}{k^k-k} - \frac{1}{3k} \geq z^{k-1}\] for all $k \geq 11$. Recalling that $\vep < \gamma/100$ and $\gamma \leq 1$, we obtain \[ \frac{1}{1 + \vep} \cdot \frac{(k-1)^{k-1}}{k^k-k}  \ge \frac{100}{101} \cdot \frac{(k-1)^{k-1}}{k^k-k}. \] Plugging in  $k=11$, we see that $$\frac{100}{101} \cdot \frac{(k-1)^{k-1}}{k^k-k} - \frac{1}{3k} \ge z^{k-1}.$$ By (\ref{eq:e_inverse1}) and the fact that $k \geq 11$, \begin{align*}
        \frac{100}{101} \cdot \frac{(k-1)^{k-1}}{k^k-k} &\geq \frac{100}{101} \cdot \frac{1}{ek}.
    \end{align*} So, it suffices to show that \[ \frac{100}{101} \cdot \frac{1}{ek} - \frac{1}{3k} \geq z^{k-1}\] for all $k \geq 12$. We do so by induction on $k$. For $k = 12$, it can be verified directly. For the induction step, assume that $k \ge 13$ and   \[ \frac{100}{101} \cdot \frac{1}{e(k - 1)} - \frac{1}{3 (k-1)} \geq z^{k-2}. \] Using Lemma~\ref{lem:nbhd_z1}, we have $(k-1)/k \geq 12/13 > 0.5 > z$ and this yields \begin{align*}
         \frac{100}{101} \cdot \frac{1}{ek} - \frac{1}{3k} &= \paren{\frac{100}{101} \cdot \frac{1}{e(k - 1)}- \frac{1}{3 (k-1)}} \frac{k-1}{k} \geq z^{k-2} \cdot z = z^{k-1},
    \end{align*} completing the proof.
    \end{proof}

\end{document}